\documentclass[12pt]{article}
\usepackage{amsmath,amssymb,amsthm}
\newtheorem{theorem}{Theorem}[section]
\newtheorem{corollary}[theorem]{Corollary}
\newtheorem{lemma}[theorem]{Lemma}

\numberwithin{equation}{section}
\def\F{\mathcal{F}}

\def\C{\mathcal{C}}

\def\B{\mathcal{B}}
\def\E{\mathcal{E}}
\def\F{\mathcal{F}}

\begin{document}
\title{Sumsets of the distance set in $\mathbb{F}_q^d$}
\author{
    Thang Pham\\
\small Department of Mathematics\\
\small Ecole Polytechnique Federale de Lausanne\\
\small \texttt{thang.pham@epfl.ch}        
    }
    \date{}
\maketitle
\begin{abstract}
Let $\mathbb{F}_q$ be a finite field of order $q$, where $q$ is large odd prime power.  In this paper, we  improve some recent results on the additive energy of the distance set, and on sumsets of the distance set due to Shparlinski (2016). More precisely, we prove that for $\mathcal{E}\subseteq \mathbb{F}_q^d$, if $d=2$ and  $q^{1+\frac{1}{4k-1}}=o(|\mathcal{E}|)$ then we have 
$|k\Delta_{\mathbb{F}_q}(\mathcal{E})|=(1-o(1))q$; if $d\ge 3$ and $q^{\frac{d}{2}+\frac{1}{2k}}=o(|\mathcal{E}|)$ then we have 
$|k\Delta_{\mathbb{F}_q}(\mathcal{E})|=(1-o(1))q,$
where $k\Delta_{\mathbb{F}_q}(\mathcal{E}):=\Delta_{\mathbb{F}_q}(\mathcal{E})+\cdots+\Delta_{\mathbb{F}_q}(\mathcal{E}) ~(\mbox{$k$ times}).$
\end{abstract}
\section{Introduction}
The Erd\H{o}s distance problem asks for the minimal number of distinct distances determined by a finite point set of $n$ points in the plane $\mathbb{R}^2$. In 1946, Erd\H{os} \cite{Erd} showed that a $\sqrt{n}\times \sqrt{n}$ integer lattice determines  $\Theta(n/\sqrt{\log n})$ distinct distances. From this construction, he conjectured that any set of $n$ points in $\mathbb{R}^2$ determines at least $n^{1-o(1)}$ distinct distances.  This conjecture has recently been solved by Guth and Katz \cite{guth} in 2010. They showed that a set of $n$ points in $\mathbb{R}^2$ has at least $cn/\log n$ distinct distances. For the latest developments on the Erd\H{o}s distance problem in higher dimensions and variants, see \cite{tardos, soly, fox}, and the references contained therein.

Let $\mathbb{F}_q$ be a finite field of order $q$, where $q$ is large odd prime power.  The distance function between two points $\mathbf{x}$ and $\mathbf{y}$ in $\mathbb{F}_q^d$, denoted by $||\mathbf{x}-\mathbf{y}||$, is defined as \[||\mathbf{x}-\mathbf{y}||=(x_1-y_1)^2+\cdots+(x_d-y_d)^2.\]

Although it is not a norm, the function $||\mathbf{x}-\mathbf{y}||$ has properties similar to the Euclidean norm, for example, it is invariant under orthogonal matrices and translations. For $\E\subseteq\mathbb{F}_q^d$, we define the set of distances determined by points in $\E$ as
\[\Delta_{\mathbb{F}_q} (\E)=\{||\mathbf{x}-\mathbf{y}||\colon \mathbf{x}, \mathbf{y}\in \E\}\subseteq \mathbb{F}_q.\]

Bourgain, Katz, and Tao \cite{bourgain-katz-tao} made the first investigation to the prime field analogue of the Erd\H{o}s distinct distance problem. More precisely, they proved that for any set $\E\subseteq \mathbb{F}_p^2$ with $|\E|=p^\alpha$, $0<\alpha<2$, the distance set satisfies $|\Delta_{\mathbb{F}_p}(\E)|\ge |\E|^{\frac{1}{2}+\epsilon}$ for some $\epsilon>0$ depending on $\alpha$. In the case $|\E|\ll p^{15/11}$, Stevens and de Zeeuw \cite{frank} improved this exponent to $|\E|^{8/15}$. This is the current best bound in the literature.\footnote{Here and throughout,  $X \gg Y$ means that there exists $C>0$ such that $X\ge  CY$, $X=o(Y)$ means that $X/Y\to 0$ as $q\to \infty$, where $X, Y$ are viewed as functions in $q$.}

For the case of large sets over arbitrary finite fields, the first explicit exponent for $|\Delta_{\mathbb{F}_q} (\E)|$ was given by Iosevich and Rudnev \cite{io} in 2007 by using Fourier analytic methods.
\begin{theorem}[\textbf{Iosevich-Rudnev}, \cite{io}]\label{NM Adidaphat}
For $\E\subseteq\mathbb{F}_q^d$ with $|\E|\gg q^{\frac{d}{2}}$, we have 
\[|\Delta_{\mathbb{F}_q}(\E)|\ge
 \min\left\lbrace q, \frac{|\E|}{q^{(d-1)/2}}\right\rbrace.\]
\end{theorem}

This result implies that if $|\E| \gg q^{(d+1)/2}$, then $|\Delta_{\mathbb{F}_q}(\E)|\gg q$.  Hart, Iosevich, Koh, Rudnev \cite{har} indicated that the threshold $q^{\frac{d+1}{2}}$ is the best possible in odd dimensions, at least in general fields.  The interested reader can find further results in \cite{ben, chap, copi, hanson, kod, sh}.

Recently Shparlinski  \cite{shparlinski} used character sum techniques to discover more properties of the distance sets. In particular, he studied properties of the additive energy of the distance sets, where the additive energy of the distance set corresponding to $\E$ and $\F$ in $\mathbb{F}_q^d$, which is denoted by $E_{+}^k(\E, \F)$, is defined as the cardinality of
\[\left\lbrace (\mathbf{x}_i, \mathbf{y}_i)_{i=1}^{2k}\in (\E\times \F)^{2k}\colon ||\mathbf{x}_1- \mathbf{y}_1||+\cdots+||\mathbf{x}_k-\mathbf{y}_k||=||\mathbf{x}_{k+1}-\mathbf{y}_{k+1}||+\cdots+||\mathbf{x}_{2k}-\mathbf{y}_{2k}||\right\rbrace.\]
When $\E=\F$, we will use the notation $E_{+}^k(\E)$ instead of $E_{+}^k(\E, \F)$. The first result in \cite{shparlinski} is the following theorem.
\begin{theorem}[\textbf{Shparlinski}, \cite{shparlinski}]\label{mot}
For $\E, \F\subseteq \mathbb{F}_q^d$, we have 
\[\left\vert E_{+}^2(\E, \F)-\frac{|\E|^4|\F|^4}{q}\right\vert\le q^{d-1}|\E|^3|\F|^3+q^{\frac{3d}{2}}|\E|^3|\F|^2.\]
\end{theorem}

As a consequence of Theorem \ref{mot}, the author of \cite{shparlinski} obtained the following result on a sumset of the distance set. 
 
\begin{theorem}[\textbf{Shparlinski}, \cite{shparlinski}]\label{comot}
For $\E, \F\subseteq \mathbb{F}_q^d$, we have
\[|\Delta_{\mathbb{F}_q}(\E, \F)+\Delta_{\mathbb{F}_q}(\E, \F)|\ge \frac{1}{3}\min\left\lbrace q, \frac{|\E||\F|^2}{q^{3d/2}}, \frac{|\E||\F|}{q^{d-1}}\right\rbrace,\]
where $\Delta_{\mathbb{F}_q}(\E, \F)=\left\lbrace ||\mathbf{x}-\mathbf{y}||\colon \mathbf{x}\in \E, \mathbf{y}\in \F\right\rbrace$.
\end{theorem}
\begin{corollary}[\textbf{Shparlinski}, \cite{shparlinski}]\label{nammoi-Addp}
Let $\E$ be a set in $\mathbb{F}_q^d$. Suppose that $q^{\frac{d}{2}+\frac{1}{3}}=o(|\E|)$, then  we have \[|\Delta_{\mathbb{F}_q}(\E)+\Delta_{\mathbb{F}_q}(\E)|= (1-o(1))q.\]
\end{corollary}

Note that the additive energy of sets is closely related to their combinatorial properties, for example, see \cite{o, o2, o3, o4, o5}. Moreover, some additive character sums can also be estimated via the additive energy, for instance, see \cite{13} for more details.

The main purpose of this paper is to give improvements of Theorems \ref{mot} and \ref{comot} by using methods from spectral graph theory. For the sake of simplicity of this paper, we will consider the case $\E=\F$. We will give some discussions at the end of Section $3$ for the case $\E\ne \F$. Our first result is the following. 

\begin{theorem}\label{hello1}
Let $\mathbb{F}_q$ be a finite field of order $q$ with $q\equiv 3\mod 4$. Let $k\ge 2$ be an integer, and $\E$ be a set in $\mathbb{F}_q^2$ with $|\E|\gg q$. We have
\[ \left\vert E_{+}^k(\E)- \frac{|\E|^{4k}}{q}\right\vert \ll q^{2k-1}|\E|^{2k+\frac{1}{2}}.\]
\end{theorem}
Our next theorem is a result on sumsets of the distance set.
\begin{theorem}\label{bon}
Let $\mathbb{F}_q$ be a finite field of order $q$ with $q\equiv 3\mod 4$. Let $k\ge 2$ be an integer, and $\E$ be a set in $\mathbb{F}_q^2$. Suppose that $q^{1+\frac{1}{4k-1}}=o(|\E|)$,  then we have 
\[|k\Delta_{\mathbb{F}_q}(\E)|=(1-o(1))q.\]
\end{theorem}
As consequences of Theorem \ref{hello1} and Theorem \ref{bon}, we are able to improve Theorem \ref{mot} and Corollary \ref{nammoi-Addp} in the case $d=2$.
\begin{corollary}\label{hai}
Let $\mathbb{F}_q$ be a finite field of order $q$ with $q\equiv 3\mod 4$. Let $\E$ be a set in $\mathbb{F}_q^2$. Suppose that $|\E|\gg q$, then we have 
\[\left\vert E_{+}^2(\E)-\frac{|\E|^8}{q}\right\vert \ll q^3|\E|^{9/2}.\]
\end{corollary}
\begin{corollary}\label{hai1}
Let $\mathbb{F}_q$ be a finite field of order $q$ with $q\equiv 3\mod 4$. Let $\E$ be a set in $\mathbb{F}_q^2$. Suppose that $q^{8/7}=o(|\E|)$,  then we have 
\[|\Delta_{\mathbb{F}_q}(\E)+\Delta_{\mathbb{F}_q}(\E)|=(1-o(1))q.\]
\end{corollary}
When $\E$ is a subset in $\mathbb{F}_q^d$ with $d\ge 3$, by using the same techniques, we obtain a similar result as follows.
\begin{theorem}\label{2017}
Let $\mathbb{F}_q$ be a finite field of order $q$. Let $k\ge 2$ be an integer, and $\E$ be a set in $\mathbb{F}_q^d$, $d\ge 3$.  We have the following
\[ \left\vert E_{+}^k(\E)- \frac{|\E|^{4k}}{q}\right\vert \ll q^{dk}|\E|^{2k}.\]
\end{theorem}
As an application of Theorem \ref{2017}, we are able to improve Corollary \ref{nammoi-Addp} in the case $d\ge 3$.
\begin{theorem}\label{bon1}
Let $\mathbb{F}_q$ be a finite field of order $q$. Let $k\ge 2$ be an integer, and $\E$ be a set in $\mathbb{F}_q^d$ with $d\ge 3$. Suppose that $q^{\frac{d}{2}+\frac{1}{2k}}=o(|\E|)$,  then we have 
\[|k\Delta_{\mathbb{F}_q}(\E)|=(1-o(1))q.\]
\end{theorem}
The rest of this paper is organized as follows: in Section $2$, we recall some graph-theoretic tools; proofs of Theorems \ref{hello1}, \ref{bon}, \ref{2017}, and \ref{bon1} are given in Section $3$. 
\section{Graph-theoretic tools}
\subsection{Expander mixing lemma}
For a graph $G$ of order $n$, let $\gamma_1 \geq \gamma_2 \geq \ldots \geq \gamma_n$ be
the eigenvalues of its adjacency matrix. The quantity $\gamma (G) = \max
\{\gamma_2, - \gamma_n \}$ is called the second eigenvalue of $G$. A graph $G
= (V, E)$ is called an $(n, d, \gamma)$-graph if it is $d$-regular, has $n$
vertices, and the second eigenvalue of $G$ is at most $\gamma$. 

Suppose that $B$ and $C$ are two multi-sets of vertices in an $(n, d, \gamma)$-graph. Let $m_X(x)$ denote the multiplicity of $x$ in $X$, and $e_m(B, C)$ be the number of edges with multiplicity between $B$ and $C$ in $G$, by multiplicity we mean that  if there is an edge between $b\in B$ and $c\in C$, then this edge will be counted $m_B(b)\cdot m_C(c)$ times in $e_m(B, C)$. Recently, Hanson et al. \cite{hanson} gave the following estimate on $e_m(B, C)$ in an $(n, d, \gamma)$-graph.
\begin{lemma}[\cite{hanson}]\label{expander}
Let $G=(V, E)$ be an $(n,d, \gamma)$-graph. The number of edges between two multi-sets of vertices $B$ and $C$ in $G$ satisfies:
\[\left\vert e_m(B, C)-\frac{d\left(\sum_{b\in B}m_B(b)\right)\left(\sum_{c\in C}m_C(c)\right)}{n}\right\vert\le \gamma\sqrt{\sum_{b\in B}m_B(b)^2}\sqrt{\sum_{c\in C}m_C(c)^2},\] where $m_X(x)$ is the multiplicity of $x$ in $X$.
\end{lemma}
\subsection{Sum-product graphs}
The sum-product graph $SP_{q, d}$ is defined as follows. The vertex set of the
sum-product graph $SP_{q, d}$ is the set $V (SP_{q, d})
=\mathbb{F}_q^d \times \mathbb{F}_q$. Two vertices $U = (\mathbf{a}, b)$ and
$V = (\mathbf{c}, d) \in V (SP_{q, d})$ are connected by an edge, $(U, V)
\in E (SP_{q, d})$, if and only if $\mathbf{a}\cdot \mathbf{c}=b+d$. Vinh \cite{vinhajm} proved the following lemma on the $(n,d,\gamma)$ form of $SP_{q, d}$.

\begin{lemma}[\textbf{Vinh}, \cite{vinhajm}]\label{sp-graph-lemma}
  For any $d \geq 1$, the sum-product graph $S
  P_{q, d}$ is an \[(q^{d + 1}, q^d, \sqrt{2 q^d})-\mbox{graph}.\]
\end{lemma}

\section{Proofs of Theorems \ref{hello1}, \ref{bon}, \ref{2017}, and \ref{bon1}}
For $\mathcal{E}\subseteq \mathbb{F}_q^d$ and $\lambda\in \mathbb{F}_q$, we define
\[\nu_{\mathcal{E}}(\lambda)=\left\vert \left\lbrace (\mathbf{x}, \mathbf{y})\in \mathcal{E}\times \mathcal{E}\colon ||\mathbf{x}-\mathbf{y}||=\lambda\right\rbrace\right\vert.\] 

In order to prove Theorems \ref{hello1}--\ref{bon1}, we need the following lemmas, where the first one follows from the proof of \cite[Theorem 3.5]{kod}.
\begin{lemma}[\textbf{Koh-Sun}, \cite{kod}]\label{pro2}
Let $\mathbb{F}_q$ be a finite field of order $q$ with $q\equiv 3\mod 4$.  Let $\E$ be a set in $\mathbb{F}_q^2$ with $|\E|\gg q$. Then we have
\[E_{+}^1(\E)=\sum_{\lambda\in \mathbb{F}_q}\nu_{\mathcal{E}}(\lambda)^2\le  \frac{|\mathcal{E}|^4}{q}+(1+\sqrt{3})q|\mathcal{E}|^{5/2}.\]
\end{lemma}

For higher dimensional cases, the authors of \cite{kod} also proved a similar result for both cases $q\equiv 3\mod 4$ and $q\equiv 1\mod 4$, which can be found in \cite[Propositions 2.3, 2.6]{kod}

\begin{lemma}[\textbf{Koh-Sun}, \cite{kod}]\label{pro21}
Let $\E$ be a set in $\mathbb{F}_q^d$ with $d\ge 3$. Then we have
\[E_{+}^1(\E)=\sum_{\lambda\in \mathbb{F}_q}\nu_{\mathcal{E}}(\lambda)^2\le  \frac{|\mathcal{E}|^4}{q}+q^d|\mathcal{E}|^{2}.\]
\end{lemma}

We will use the following lemma to prove Theorem \ref{hello1} and Theorem \ref{2017}.
\begin{lemma}\label{hello}
Let $k\ge 2$ be an integer, and $\E$ be a set in $\mathbb{F}_q^d$. We have 
\[\left\vert E_{+}^k(\E)-\frac{|\E|^{4k}}{q} \right\vert\ll q^d|\E|^2 E_{+}^{k-1}(\E) .\]
\end{lemma}
\begin{proof}
We first define two multi-sets of vertices in the sum-product graph $SP_{q, 2d}$ as follows:
\[\B:=\left\lbrace \left(-2\mathbf{x}_1, -2\mathbf{x}_2, -||\mathbf{x}_1||-||\mathbf{x}_2||-||\mathbf{x}_3-\mathbf{y}_3||-\cdots-||\mathbf{x}_k-\mathbf{y}_k||+||\mathbf{x}_{k+1}-\mathbf{y}_{k+1}||\right)\colon \mathbf{x}_i, \mathbf{y}_i\in \E\right\rbrace,\]
\[\C:=\left\lbrace \left(\mathbf{y}_1, \mathbf{y}_2, -||\mathbf{y}_1||-||\mathbf{y}_2||+||\mathbf{x}_{k+2}-\mathbf{y}_{k+2}||+\cdots+||\mathbf{x}_{2k}-\mathbf{y}_{2k}||\right)\colon \mathbf{x}_i, \mathbf{y}_i\in \E \right\rbrace.\]
For $(\mathbf{x}_i, \mathbf{y}_i)_{i=1}^{2k}\in (\E\times \E)^{2k}$, if we have 
\[
||\mathbf{x}_1- \mathbf{y}_1||+\cdots+||\mathbf{x}_k-\mathbf{y}_k||=||\mathbf{x}_{k+1}-\mathbf{y}_{k+1}||+\cdots+||\mathbf{x}_{2k}-\mathbf{y}_{2k}||,\]
then there is an edge between 
\[\left(-2\mathbf{x}_1, -2\mathbf{x}_2, -||\mathbf{x}_1||-||\mathbf{x}_2||-||\mathbf{x}_3-\mathbf{y}_3||-\cdots-||\mathbf{x}_k-\mathbf{y}_k||+||\mathbf{x}_{k+1}-\mathbf{y}_{k+1}||\right)\in \B\]
and 
\[\left(\mathbf{y}_1, \mathbf{y}_2, -||\mathbf{y}_1||-||\mathbf{y}_2||+||\mathbf{x}_{k+2}-\mathbf{y}_{k+2}||+\cdots+||\mathbf{x}_{2k}-\mathbf{y}_{2k}||\right) \in \C\]
in the sum-product graph $SP_{q, 2d}$. Therefore $E_{+}^k(\E)$ is equal to the number of edges between $\B$ and $\C$ in $SP_{q, 2d}$.  In order to apply Lemma \ref{expander}, we need to estimate upper bounds of  $\sum_{b\in \B}m_{\B}(b)^2$ and $\sum_{c\in \C}m_{\C}(c)^2$. One can check that 
\[\sum_{b\in \B}m_{\B}(b)^2\le |\E|^2 E_{+}^{k-1}(\E), ~\sum_{c\in \C}m_{\C}(c)^2\le |\E|^2 E_{+}^{k-1}(\E), ~\mbox{and}~|\B|=|\C|=|\E|^{2k}.\]
It follows from Lemmas \ref{expander} and \ref{sp-graph-lemma} that the number of edges between $\B$ and $\C$ in the sum-product graph $SP_{q, 2d}$ satisfies 
\[\left\vert E_{+}^k(\E)-\frac{|\E|^{4k}}{q}\right\vert  \ll q^d|\E|^2 E_{+}^{k-1}(\E),\]
which concludes the proof of the lemma.\end{proof}

\paragraph{Proof of Theorem \ref{hello1}:}
The proof proceeds by induction on $k$. The base case $k=2$ follows from Lemma \ref{pro2} and Lemma \ref{hello} with $d=2$. Suppose that the claim holds for $k-1\ge 2$, we show that it also holds for $k$. Indeed, it follows from Lemma \ref{hello} with $d=2$ that 
\begin{equation}\label{1011} \left\vert E_{+}^k(\E)-\frac{|\E|^{4k}}{q}\right\vert \ll q^2|\E|^2 E_{+}^{k-1}(\E) .\end{equation}

By induction hypothesis, we have 
\begin{equation}\label{10111} E_{+}^{k-1}(\E)\ll \frac{|\E|^{4(k-1)}}{q}+q^{2(k-1)-1}|\E|^{2(k-1)+\frac{1}{2}}.\end{equation}

Putting (\ref{1011}) and (\ref{10111}) together gives us 
\[\left\vert E_{+}^k(\E)- \frac{|\E|^{4k}}{q}\right\vert \ll q^{2k-1}|\E|^{2k+\frac{1}{2}},\]
which ends the proof of the theorem. $\square$
\paragraph{Proof of Theorem \ref{bon}:}
For each $\lambda\in \mathbb{F}_q$, let $N_\lambda$ be the number of tuples $(\mathbf{x}_1, \mathbf{y}_1,\ldots, \mathbf{x}_k, \mathbf{y}_k)$ in $\E^{2k}$ satisfying $||\mathbf{x}_1-\mathbf{y}_1||+||\mathbf{x}_2-\mathbf{y}_2||+\cdots+||\mathbf{x}_k-\mathbf{y}_k||=\lambda$. We have $\sum_{\lambda\in \mathbb{F}_q}N_{\lambda}=|\E|^{2k}$. It is easy to check that $\sum_{\lambda\in \mathbb{F}_q}N_{\lambda}^2=E_+^k(\E)$. By applying the Cauchy-Schwarz inequality, we obtain the following 
\[\sum_{\lambda\in \mathbb{F}_q}N_{\lambda}\le \sqrt{|k\Delta_{\mathbb{F}_q}(\E)|}\left(E_+^k(\E)\right)^{1/2}.\]
This implies that 
\[|k\Delta_{\mathbb{F}_q}(\E)|\ge \frac{|\E|^{4k}}{E_+^k(\E)}.\]
Thus the theorem follows immediately from Theorem \ref{hello1}. $\square$
\paragraph{Proof of Theorem \ref{2017}:}
The proof of Theorem \ref{2017} is as similar as that of Theorem \ref{hello1} except that we use  Lemma \ref{pro21} instead of Lemma \ref{pro2}. $\square$
\paragraph{Proof of Theorem \ref{bon1}:}
The proof of Theorem \ref{bon1} is as similar as that of Theorem \ref{bon} except that we use Theorem \ref{2017} instead of Theorem \ref{hello1}. $\square$\\
{\bf Remarks:}
We conclude this paper with some discussions on $E_{+}^2(\E, \F)$ for $\E, \F\subseteq \mathbb{F}_q^d$ satisfying $|\E|<|\F|$.  The main steps in our approach are Lemma \ref{hello} and upper bounds of $E_{+}^1(\E, \F)$. For two sets $\E$ and $\F$ in $\mathbb{F}_q^2$ with $q\equiv 3\mod 4$, it has been shown in \cite{kod} that 
\begin{equation}\label{cuoingay1}E_{+}^1(\E, \F)\ll \frac{|\E|^2|\F|^2}{q}+q|\E|^{3/2}|\F| ~\mbox{for $d=2$},\end{equation}
and 
\begin{equation}\label{cuoingay2}E_{+}^1(\E, \F)\ll \frac{|\E|^2|\F|^2}{q}+q^{\frac{d-1}{2}}|\E|^{2}|\F| ~\mbox{for odd $d\ge 3$}.\end{equation}
For $\E, \F\subseteq \mathbb{F}_q^d$, one can follow the proof of Lemma \ref{hello} to obtain the following 
\begin{equation}\label{cuoingay3}\left\vert E_{+}^k(\E, \F)-\frac{|\E|^{2k}|\F|^{2k}}{q} \right\vert\ll q^d|\E||\F| E_{+}^{k-1}(\E, \F).\end{equation}
If we put (\ref{cuoingay1}), (\ref{cuoingay2}), and (\ref{cuoingay3}) together, then we have
\[\left\vert E_{+}^2(\E, \F)-\frac{|\E|^4|\F|^4}{q}\right\vert\le q|\E|^3|\F|^3+q^3|\E|^{\frac{5}{2}}|\F|^2 ~\mbox{for $d=2$},\]
\begin{equation}\left\vert E_{+}^2(\E, \F)-\frac{|\E|^4|\F|^4}{q}\right\vert\le q^{d-1}|\E|^3|\F|^3+q^{\frac{3d-1}{2}}|\E|^3|\F|^2 ~\mbox{for odd}~d\ge 3 \nonumber.\end{equation}
These results are also improvements of Theorem \ref{mot}.
\section{Acknowledgments}
The author would like to thank three reviewers for valuable comments and suggestions which improved the presentation of this paper considerably. The author would like to thank Prof. Doowon Koh for pointing out the necessary conditions of Lemma \ref{pro2}, Theorem \ref{hello1}, and Theorem \ref{bon}.
The author was partially supported by Swiss National Science Foundation grants 200020-162884 and 200021-175977.

\end{document}